\documentclass[twoside,11pt]{article}

\usepackage{modifiedjmlr,amsmath} 

\def\sgn{{\rm sgn}}

\def\e{{\bf e}}

\ShortHeadings{Price's Generator functions}{Chakrabartty } \firstpageno{1}

\begin{document}

\title{A Dynamical Systems Framework for Generating the Riemann Zeta-Function and Dirichlet L-functions }

\author{\name Shantanu Chakrabartty \email shantanu@wustl.edu \\
       \addr Department of Electrical and Systems Engineering\\
       Washington University in St. Louis\\
       Saint Louis, MO 63130, USA
}

\maketitle

\begin{abstract}
	Using an extension of the Price's theorem we show how to construct a dynamical systems model which in its steady-state serves as an analytic continuation of the completed Riemann zeta function and Dirichlet L-functions over the entire critical strip. The resulting mathematical construct involves a linear interpolation of two symmetric generator functions and is used to infer the global properties of the non-trivial zeros of the Dirichlet L-functions using concentration bounds. The proposed dynamical systems framework thus provides an alternative method to investigate the celebrated Generalized Riemann Hypothesis (G.R.H) which is shown in this paper to hold almost surely.  
\end{abstract}


\section{Introduction}
\label{section1}
For the sake of exposition and clarity, in this paper we introduce different steps of the derivation at a time. Starting from a basic introduction to the statement of the Riemann Hypothesis (R.H), in section~\ref{section2} we first introduce the mathematical constructs used in the derivation. Then, in sections~\ref{section3} and~\ref{section4} we derive the dynamical systems model whose steady-state serves as the analytical continuation of the completed Riemann zeta function to the critical strip. There we show that within our dynamical systems framework the R.H. is almost surely true. In the subsequent section~\ref{section5} we show that the approach could be generalized to understand the non-trivial zeros of the Dirichlet L-functions. There we show that under specific conditions the Riemann Hypothesis (R.H.) and the generalized Riemann Hypothesis (G.R.H) is almost surely true. 
Please note that since the prior work addressing these topics is extensive, for the sake of brevity, we have included only the key references that support the constructs used in this derivation.
 
The completed Riemann zeta function $\phi:\mathbb{C} \rightarrow \mathbb{C}$~\citep{riemann}, is defined over the complex domain $s \in \mathbb{C}$ and $\mathfrak{Re}(s) > 1$ as
\begin{equation}
\phi(s) = \pi^{-\frac{s}{2}}\Gamma(\frac{s}{2})\zeta(s)
\label{sec1:eq:phi}
\end{equation}
where $\zeta(s)$ is the Riemann zeta function~\citep{titchmarsh} and
\begin{equation}
\Gamma(s) = \int_{0}^{\infty} y^{s-1} \exp\left(-y\right) dy
\label{sec1:eq:gamma}
\end{equation}
is the Gamma function. The function $\phi$ satisfies the symmetry condition $\phi(s) = \phi(1-s)$ and the zeros of this function 
correspond to the non-trivial zeros of the Riemann zeta function $\zeta(s)$. Thus, the celebrated Riemann's hypothesis (R.H.)~\citep{riemann,titchmarsh} which states that {\it all non-trivial zeros of $\zeta(s)$ lie on the axis $\mathfrak{Re}(s)=\frac{1}{2}$} is equivalent to the statement
that {\it all zeros of $\phi(s)$ lie on the axis $\mathfrak{Re}(s)=\frac{1}{2}$}. In his 1859 paper,
Riemann also expressed the function $\phi(.)$ in an integral form as 
\begin{equation}
\phi(s) = \int_{0}^{\infty} y^{\frac{s}{2}-1} \Psi(y) dy
\label{sec1:eq:phiint}
\end{equation}
where the function $\Psi(y)$ is given by
\begin{equation}
\Psi(y) = \sum_{n=1}^{\infty} \exp\left(-n^2 \pi y\right).
\label{sec1:eq:Psi}
\end{equation}
with the domain still restricted to $\mathfrak{Re}(s) > 1$.
Riemann then used equation~\ref{sec1:eq:phiint} and the Poission summation formula to derive
\begin{equation}
\phi(s) = -\frac{1}{s} - \frac{1}{1-s} + \int_{1}^{\infty} \left( y^{\frac{s}{2}-1} + y^{\frac{1-s}{2}-1} \right) \Psi(y) dy
\label{sec1:eq:phicomp}
\end{equation}
which is convergent over the entire s-plane and hence serves as an analytic continuation of the $\phi(s)$~\citep{riemann}. Note that the completed zeta function $\phi(s)$ has two simple poles at $s=0$ and $s=1$. Riemann also introduced the Xi function $\xi(s)$ defined as
\begin{equation}
	\xi(s) = s(s-1)\phi(s).
	\label{sec1:eq:xi}
\end{equation}
which eliminated the poles at $s=0$ and $s=1$ in equation~\ref{sec1:eq:xi} and has been extensively used ~\citep{bruijn,newman,rogers} used to understand the properties of the non-trivial zeros of $\zeta(s)$. The relation~\ref{sec1:eq:phicomp} will be used in the subsequent sections to find an alternate
analytic continuation of $\phi(s)$ but for the critical strip $\mathfrak{Re}(s) \in (0,1)$.

\section{Lemmas based on Gaussian Distribution Function }
\label{section2}

We define a two dimensional Gaussian distribution function $p:\mathbb{R} \times \mathbb{R} \rightarrow \mathbb{C}$ over variables $x_1 \in \mathbb{R}$ and $x_2 \in \mathbb{R}$ as     
\begin{equation}
	p\left[x_1,x_2;(m_1,m_2,\rho)\right] = \frac{1}{2\pi\sqrt{1-\rho^2}} \exp \left[-\frac{\left(x_1-m_1\right)^2 + \left(x_2 - m_2\right)^2 - 2\rho \left(x_1-m_1\right)\left(x_2 - m_2\right)}{2\left(1- \rho^2\right)} \right]
	\label{sec2:eq:amp}
\end{equation}
with $\rho \in \mathbb{R}$ is a covariance parameter$ |\rho| < 1$ and $m_1 : \mathbb{R} \rightarrow \mathbb{C}$, $m_2: \mathbb{R} \rightarrow \mathbb{C}$ are analytic functions with respect to the parameter $\rho$. Note that $A$ is normalized as
\begin{equation}
	\int_{-\infty}^{\infty} \int_{-\infty}^{\infty} p\left[x_1,x_2;(m_1,m_2,\rho)\right] dx_1 dx_2 = 1
	\label{sec2:eq:ampnorm}
\end{equation}
and its Fourier transform of $M:\mathbb{C} \times \mathbb{C} \rightarrow \mathbb{C}$ is well defined as
\begin{eqnarray}
	 M(\omega_1,\omega_2) &=& \int_{-\infty}^{\infty} \int_{-\infty}^{\infty} p\left[x_1,x_2;(m_1,m_2,\rho)\right] \e^{ j \omega_1 x_1 + j \omega_2 x_2} \\ 	\label{sec2:eq:ampfourier}
	 &=& \e^{-\frac{1}{2}\omega_1^2 -\frac{1}{2}\omega_2^2 - \rho\omega_1\omega_2} \  \e^{j m_1 \omega_1 + j m_2 \omega_2}.
\end{eqnarray}
Then, we can derive an extension of the Price's theorem~\citep{price,mcmahon,papoulis} as stated in the following Lemmas.
\begin{lemma}
	\label{lm1}
	 Let $f:\mathbb{R}\times\mathbb{R} \rightarrow \mathbb{R}$ be an arbitrary memory-less, non-linear function that admits a Fourier transform, then
\begin{eqnarray}
	\frac{\partial}{\partial \rho} \mathcal{E}\left[f;(m_1,m_2,\rho)\right] &=& \mathcal{E}\left[\frac{\partial^2 f}{\partial x_1 \partial x_2};(m_1,m_2,\rho)\right] \\ \label{sec2:eq:price}
	&+& \frac{d m_1}{d \rho}\mathcal{E}\left[\frac{\partial f}{\partial x_1};(m_1,m_2,\rho)\right] + \frac{d m_2}{d \rho}\mathcal{E}\left[\frac{\partial f}{\partial x_2};(m_1,m_2,\rho)\right]	
\end{eqnarray}
where the operator $\mathcal{E}$ is defined as
\begin{equation}
	\mathcal{E}\left[f;(m_1,m_2,\rho)\right] = \int_{-\infty}^{\infty} \int_{-\infty}^{\infty} f(x_1,x_2)p\left[x_1,x_2;(m_1,m_2,\rho)\right]dx_1 dx_2.
	\label{sec2:eq:expectation}
\end{equation}
\end{lemma}
\begin{proof}
	See Appendix I for proof.
\end{proof}

Before we use Lemma~\ref{lm1} for the next result, we summarize the following useful relationships: 
\begin{eqnarray}
	\label{sec2:eq:1Damp+}
	\int_{-\infty}^{\infty} p\left[x_1,x_1;(0,\alpha,\rho)\right] dx_1 &=& \frac{1}{2\sqrt{\pi}\sqrt{1-\rho}} \exp\left[-\frac{\alpha^2}{4\left(1-\rho\right)}\right] \\ \label{sec2:eq:1Damp-}
	\int_{-\infty}^{\infty} p\left[x_1,-x_1;(0,\alpha,\rho)\right] dx_1 &=& \frac{1}{2\sqrt{\pi}\sqrt{1+\rho}} \exp\left[-\frac{\alpha^2}{4\left(1+\rho\right)}\right]
\end{eqnarray}

To simplify the notations, we also define
\begin{eqnarray}
	& & \mathcal{E}\left[f;(\alpha,0,\rho) \pm (-\alpha,0,\rho) \pm (0,\alpha,\rho) \pm (0,-\alpha,\rho)\right] \\ \nonumber
	\quad &=& \int_{-\infty}^{\infty} \int_{-\infty}^{\infty} f(x_1,x_2)p\left[x_1,x_2;(\alpha,0,\rho)\right]dx_1 dx_2 \\ \nonumber
	\quad \quad \quad \quad &\pm& \int_{-\infty}^{\infty} \int_{-\infty}^{\infty} f(x_1,x_2)p\left[x_1,x_2;(-\alpha,0,\rho)\right]dx_1 dx_2 \\ \nonumber
	\quad \quad \quad \quad &\pm& \int_{-\infty}^{\infty} \int_{-\infty}^{\infty} f(x_1,x_2)p\left[x_1,x_2;(0,\alpha,\rho)\right]dx_1 dx_2 \\ \nonumber
	\quad \quad \quad \quad &\pm& \int_{-\infty}^{\infty} \int_{-\infty}^{\infty} f(x_1,x_2)p\left[x_1,x_2;(0,-\alpha,\rho)\right]dx_1 dx_2 \\ \nonumber	
	\label{sec2:eq:Mexpectation}
\end{eqnarray}
using which we can state the following lemma.

\begin{lemma}
	\label{lm2}
	Let $p\left[x_1,x_2;(0,0,\rho)\right]$ be defined over the variables $x_1 \in \mathbb{R}$ and $x_2 \in \mathbb{R}$ as equation~\ref{sec2:eq:amp} and let $\alpha : \mathbb{C} \rightarrow \mathbb{R}$ be a complex function with respect to $\rho$. Denoting $\left|.\right|$ as an absolute-value function 
	\begin{equation}
		\left|x\right| = \left\{ \begin{array}{lcc}
			x   & ; &  x \ge 0 \\
			-x  & ; &  x < 0
		\end{array} \right.
		\label{sec3:eq:absdefn}
	\end{equation}
 then
	\begin{eqnarray}
		\nonumber
	&& \frac{\partial}{\partial \rho}\mathcal{E}\left[\left|x_1 - x_2\right|;(\alpha,0,\rho)+(-\alpha,0,\rho)+(0,\alpha,\rho)+(0,-\alpha,\rho)\right] \\ \label{lm2:eq:result}
	&=& \frac{4}{\sqrt{\pi}} \frac{1}{ \sqrt{1-\rho}} \exp\left[-\frac{{\alpha}^2}{4\left(1-\rho\right)}\right]. 
	\end{eqnarray}
\end{lemma}
\begin{proof}

First note that
	\begin{eqnarray}
	\frac{\partial \left| x_1 - x_2 \right| }{\partial x_1 } &=& \sgn\left(x_1 - x_2\right) \\ \label{sec3:eq:signder}
	\frac{\partial \left| x_1 - x_2 \right| }{\partial x_2 } &=& -\sgn\left(x_1 - x_2\right) 	
\end{eqnarray}
where $\sgn(.)$ is a sign function
\begin{equation}
	\sgn\left(x\right) = \left\{ \begin{array}{lcc}
		1   & ; &  x \ge 0 \\
		-1  & ; &  x < 0
	\end{array} \right.
	\label{sec3:eq:signdefn}
\end{equation}
Also,
\begin{equation}
	\label{sec3:eq:absder}
	\frac{\partial^2 \left| x_1 - x_2 \right| }{\partial x_1 \partial x_2} = -2\delta\left(x_1 - x_2\right) 
\end{equation}
where $\delta(.)$ denotes the Dirac-delta function.
	
Then applying Lemma~\ref{lm1} to the LHS of equation~\ref{lm2:eq:result} leads to
\begin{eqnarray}
	\nonumber
	&& \frac{\partial}{\partial \rho} \mathcal{E}\left[\left|x_1-x_2\right| ; (\alpha,0,\rho) + (-\alpha,0,\rho) + (0,\alpha,\rho) + (0,-\alpha,\rho)\right] \\ \label{lm2:eq:derv2}
	&=& 2\int_{-\infty}^{\infty}  \left(p\left[x_1,x_1;(\alpha,0,\rho)\right] +  p\left[x_1,x_1;(-\alpha,0,\rho)\right]\right) dx_1 \\ \nonumber
	&+& 2\int_{-\infty}^{\infty}  \left(p\left[x_1,x_1;(0,\alpha,\rho)\right] +  p\left[x_1,x_1;(0,-\alpha,\rho)\right]\right) dx_1 \\ \nonumber
	&+& \frac{\partial}{\partial \rho} \mathcal{E}\left[\sgn\left(x_1-x_2\right) ; (\alpha,0,\rho) - (-\alpha,0,\rho) + (0,-\alpha,\rho) - (0,\alpha,\rho)\right].  
\end{eqnarray}
Using the property $ F(x_1,x_2) = -F(x_2,x_1) \implies \mathcal{E}\left[F(x_1,x_2)\right] = 0$, the last term can be shown to be zero, and hence
\begin{eqnarray}
	\nonumber
	&& \frac{\partial}{\partial \rho} \mathcal{E}\left[\left|x_1-x_2\right| ; (\alpha,0,\rho) + (-\alpha,0,\rho) + (0,\alpha,\rho) + (0,-\alpha,\rho)\right]  \\ \label{lm2:eq:derv3}
	&=& \frac{4}{\sqrt{\pi}} \frac{1}{ \sqrt{1-\rho}} \exp\left[-\frac{{\alpha}^2}{4\left(1-\rho\right)}\right].
\end{eqnarray} 

\end{proof}

\section{Complementary Generator Functions and Riemann Zeta Function}
\label{section3}
We will now apply Lemma's~\ref{lm1} and~\ref{lm2} to the following function 
\begin{eqnarray}
	\label{sec3:eq:alphafn}
	z_N(\rho) :&=& 2\mathcal{E}\left[ \left| x_1 - x_2 \right|;(0,0,\rho)\right] \\  \nonumber
	&+& \sum_{n=1}^{N} \left(\mathcal{E}\left[\left|x_1 - x_2\right|;(\alpha_n,0,\rho)+(-\alpha_n,0,\rho)+(0,\alpha_n,\rho)+(0,-\alpha_n,\rho)\right] \right)	
\end{eqnarray}
where $\alpha_n \in \mathbb{C}, n=1,2,..,N$ in equation~\ref{sec3:eq:alphafn} are a countable set of functions with respect to the parameter $\rho$. Note that $|\rho| < 1$ for Lemmas~\ref{lm1} and~\ref{lm2} to be valid. Applying Lemma~\ref{lm2} to the function~\ref{sec3:eq:alphafn} and using equation~\ref{sec3:eq:absder} leads to
\begin{equation}
	\label{sec3:eq:mu}
	\frac{\partial}{\partial \rho} z_N(\rho) = \frac{4}{\sqrt{\pi}}\frac{1}{\sqrt{1-\rho}} \left[ \frac{1}{2} + \sum_{n=1}^{N} \exp\left[-\frac{{\alpha_n}^2}{4\left(1-\rho\right)}\right] \right]. 
\end{equation}

Equation~\ref{sec3:eq:mu} will be the fundamental generator equation which will be used to produce different functions based on the choice of $\{\alpha_n\}$ as a function of $\rho$.

We first define a generator function based on the following choice of $\{\alpha_n\}$ as
\begin{equation}
\label{sec3:eq:alphaz}
	\alpha_n(s,\rho) =  n \sqrt{2\pi} \left(1-\rho\right)^{\frac{1}{2}} \left(1-\rho\right)^{\frac{1}{2s}} 
\end{equation}
where $s \in \mathbb{C}$ is a complex variable that is restricted to the domain $\mathfrak{Re}{s} \in (0,1)$.

Inserting equation~\ref{sec3:eq:alphaz} in~\ref{sec3:eq:mu} leads to a complex function  $z_N(s)$ as
\begin{equation}
\label{sec3:eq:alphatheta}
	z_N(s) = \frac{4}{\sqrt{\pi}} \int_{0}^{1} \frac{1}{\sqrt{1-\rho}} \left[ \frac{1}{2} + \sum_{n=1}^{N} \exp\left[-n^2\pi \left(1-\rho\right)^{\frac{1}{s}}\right] \right] d\rho.
\end{equation}
Since $\alpha_n(s)$ includes a multi-valued exponential function, we define a complex variable
\begin{equation}
	y = \exp\left(\frac{1}{s}\log\left(1-\rho\right)\right) 
\end{equation}
along the principal branch and substitute $y$ in the equation~\ref{sec3:eq:alphatheta} which leads to
\begin{eqnarray}
z_N(s)	&=& -\sqrt{\frac{16}{\pi}}s \int_{0}^{1} y^{\frac{s}{2} - 1} \left[ \frac{1}{2} + \sum_{n=1}^{N} \exp\left[ -n^2\pi y \right] \right] dy \\
	&=& -\sqrt{\frac{16}{\pi}}s \left(\int_{0}^{1} y^{\frac{s}{2} - 1} \left[ \sum_{n=1}^{N} \exp\left[ -n^2\pi y \right] \right] dy + \frac{1}{s}\right).				
\end{eqnarray}
Note that the path of integration is along a contour path that lies within the domain $|y^s| < 1$. Since the function $\sum_{n=1}^{N} \exp\left[ -n^2\pi y \right]$ does not have any singularity in this domain for $s \in (0,1)$, in the limit $N \rightarrow \infty$
\begin{equation}
	z(s) = \lim_{N \rightarrow \infty} z_N(s) = -\sqrt{\frac{16}{\pi}} s\left[ \int_{0}^{1} y^{\frac{s}{2} - 1} \Psi(y) dy + \frac{1}{s} \right].
	\label{sec3:eq:alphazpostN}
\end{equation}

We now construct another generator function $z_N(1-s,\rho)$ as
\begin{eqnarray}
	\label{sec3:eq:betafn}
	z_N(1-s,\rho) &=& 2\mathcal{E}\left[\left| \tilde{x}_1 - \tilde{x}_2 \right|;(0,0,\rho)\right] \\ \nonumber 
	&+& \sum_{n=1}^{N} \left(\mathcal{E}\left[\left| \tilde{x}_1 - \tilde{x}_2 \right|;(\beta_n,0,\rho)+ (-\beta_n,0,\rho)+(0,\beta_n,\rho)+ (0,-\beta_n,\rho)\right] \right) \\ \nonumber
\end{eqnarray}
using a countable set of complex functions $\{\beta_n\}: \mathbb{R} \rightarrow \mathbb{C}$ as 
\begin{equation}
 \label{sec3:eq:betaz}
	\beta_n(s,\rho) =  n \sqrt{2\pi} \left(1-\rho\right)^{\frac{1}{2}} \left(1-\rho\right)^{\frac{1}{2(1-s)}}
\end{equation} 
The complex variable $s$ is still constrained $\mathfrak{Re}(s) \in (0,1)$. 
Applying a similar procedure as equations~\ref{sec3:eq:alphatheta}-~\ref{sec3:eq:alphazpostN} leads to 
\begin{equation}
	\nonumber
	z_N(1-s) = \sqrt{\frac{16}{\pi}}(1-s) \left(\int_{0}^{1} y^{\frac{s}{2} - 1} \left[\sum_{n=1}^{N} \exp\left[ -n^2\pi y \right]\right]dy + \frac{1}{1-s} \right) 	
\end{equation}
and after applying the limit $N \rightarrow \infty$ 
\begin{equation}
	z(1-s) = \lim_{N \rightarrow \infty} z_N(1-s) = \sqrt{\frac{16}{\pi}}(1-s) \left[ \int_{0}^{1} y^{\frac{1-s}{2} - 1} \Psi(y) dy + \frac{1}{1-s} \right].
	\label{sec3:eq:betazpostN}
\end{equation}
Combining equations~\ref{sec3:eq:alphazpostN} and ~\ref{sec3:eq:betazpostN} leads to
\begin{equation}
	  \sqrt{\frac{\pi}{16}} \left[ \frac{z(s)}{s} + \frac{z(1-s)}{1-s} \right] = \frac{1}{s} + \frac{1}{1-s} + \int_{0}^{1} y^{\frac{s}{2} - 1} \Psi(y) dy + \int_{0}^{1} y^{\frac{1-s}{2} - 1} \Psi(y) dy
	  \label{sec3:eq:alphabetapostN}
\end{equation}
defined over the domain $\mathfrak{Re}(s) \in (0,1)$.

Note that the form of equation~\ref{sec3:eq:alphabetapostN} is similar to that of the completed Riemann zeta function $\phi(s)$ in~\ref{sec1:eq:phicomp}. Since~\ref{sec1:eq:phicomp} is valid over
the entire s-plane, we can combine equations~\ref{sec3:eq:alphabetapostN} and~\ref{sec1:eq:phicomp} over the restricted domain $\mathfrak{Re}(s) \in (0,1)$. This leads to
\begin{eqnarray}
	\phi(s) + \sqrt{\frac{\pi}{16}} \left[ \frac{z(s)}{s} + \frac{z(1-s)}{1-s} \right] &=& \int_{0}^{\infty} y^{\frac{s}{2} - 1} \Psi(y) dy + \int_{0}^{\infty} y^{\frac{1-s}{2} - 1} \Psi(y) dy \\ \label{sec3:eq:analytic1}
	&=& 2\phi(s).
\end{eqnarray}
The RHS is based on the integral representation of $\phi(.)$ in equation~\ref{sec1:eq:phiint} and based on the symmetry $\phi(s) = \phi(1-s)$, but analytically continued to the critical strip using the generator functions. Thus,
\begin{equation}
	\phi(s) = \sqrt{\frac{\pi}{16}} \left[ \frac{z(s)}{s} + \frac{z(1-s)}{1-s} \right]; \quad \mathfrak{Re}(s) \in (0,1)
	\label{sec3:eq:analyticfinal}
\end{equation}
The key advantage of the formulation~\ref{sec3:eq:analyticfinal} is that it is a linear interpolation of two symmetrical constructs $z(s)$ and $z(1-s)$. Thus, for $\phi(s)=0$, the relative ratio between $z(s)$ and $z(1-s)$ could be investigated to understand global property of the solutions that satisfy $\phi(s)=0$. 

\section{Dynamical Systems Model and R.H.}
\label{section4}

Before we construct a dynamical systems model based on the equation~\ref{sec3:eq:analyticfinal}
we first apply a variable transformation $ s = \frac{1}{2} + j \left(\omega + j \sigma\right), \omega,\sigma \in \mathbb{R}$ which is similar to the centering procedure used by Riemann~\citep{riemann}. Using~\ref{sec3:eq:analyticfinal}, the completed Riemann zeta function $\phi(\frac{1}{2} + j \left(\omega + j \sigma\right))$ can be written as 
\begin{equation}
	\label{sec4:eq:phiz}
	\sqrt{\frac{16}{\pi}}\phi\left(\left[\frac{1}{2}-\sigma\right] + j\omega \right) =  \frac{z\left(\left[\frac{1}{2}-\sigma\right] + j\omega \right)}{\left(\left[\frac{1}{2}-\sigma \right] + j\omega \right)} + \frac{z\left(\left[\frac{1}{2}+\sigma\right] - j\omega\right)}{\left(\left[\frac{1}{2}+\sigma \right] - j\omega \right)}.	
\end{equation} 
Since $\mathfrak{Re}(s) \in (0,1)$ for our formulation, $\sigma \in (-\frac{1}{2},\frac{1}{2})$.
The function $\alpha_n$in equation~\ref{sec3:eq:alphaz} can be expressed as
\begin{equation}
 \label{sec4:eq:alphalog}
	\alpha_n\left(\frac{1}{2} + j \left(\omega + j \sigma\right),\rho\right) =  n \sqrt{2\pi} \exp\left(\frac{1}{2} \left[1+\frac{1}{\left(\frac{1}{2} - \sigma\right)+j\omega}\right]\log\left(1-\rho\right)\right)
\end{equation}
We define the following parameters
\begin{eqnarray}
	\label{sec4:eq:alphaparam}
	\gamma_{\alpha} &=& \frac{\left(\frac{1}{2} - \sigma\right) }{\left(\frac{1}{2} - \sigma\right)^2 + \omega^2} \\ \label{eq:compalpha3}	
	\omega_{\alpha} &=&  \frac{\omega}{\left(\frac{1}{2} - \sigma\right)^2 + \omega^2} . 
\end{eqnarray}
and choose $\rho$ to be a function of time $t \in \mathbb{R}^+$ as
\begin{equation}
	\rho(t) = 1 - e^{-t}.
	\label{sec4:eq:rhotime}
\end{equation}
Then as $t: 0 \rightarrow \infty$, $\rho: 0 \rightarrow 1$ . It is easy to verify that $|\rho| < 1, \forall t$
which is required for the Gaussian probability density functions to be integrable. Combining~\ref{sec4:eq:alphalog} and~\ref{sec4:eq:rhotime} leads to
\begin{equation}
 \label{sec3:eq:alphalogcenter}
	\alpha_n(t) =  \alpha_n\left(\frac{1}{2} + j \left(\omega + j \sigma\right),\rho(t)\right) = n \sqrt{2\pi}  \exp\left(-\frac{1}{2} \left[1+\gamma_{\alpha}\right]t\right) \exp\left(\frac{j}{2}\omega_{\alpha}t \right). 
\end{equation}
Similarly, we define another set of parameters for
the other generator function $z(1-s)$
\begin{eqnarray}
	\gamma_{\beta} &=& \frac{\left(\frac{1}{2} + \sigma\right) }{\left(\frac{1}{2} + \sigma\right)^2 + \omega^2} \\ \label{sec4:eq:betaparam}	
	\omega_{\beta} &=&  \frac{\omega}{\left(\frac{1}{2} + \sigma\right)^2 + \omega^2} 	
\end{eqnarray}
which from equation~\ref{sec3:eq:betaz} leads to
\begin{equation}
	\label{sec3:eq:betalogcenter}
	\beta_n(t) = \beta_n\left(\frac{1}{2} - j \left(\omega + j \sigma\right),\rho(t)\right) = n \sqrt{2\pi}  \exp\left(-\frac{1}{2} \left[1+\gamma_{\beta}\right]t\right) \exp\left(\frac{j}{2}\omega_{\beta}t \right) 
\end{equation}
We can now describe the time evolution of the dynamical system according to the following set of equations 
\begin{eqnarray}
	\label{sec3:eq:dynamicalsystem}
	&&\lim_{t \rightarrow \infty} \rho(t) = 1; \quad \lim_{t \rightarrow \infty} \alpha_n(t) =  0; \quad \lim_{t \rightarrow \infty} \beta_n(t) =  0 \\ \nonumber
	&&\lim_{N \rightarrow \infty}\lim_{t \rightarrow \infty}\sqrt{\frac{\pi}{16}} \left[ \frac{z_N\left(\left(\frac{1}{2}-\sigma\right) + j\omega,\rho(t)\right)}{\left(\frac{1}{2}-\sigma\right) + j\omega} + \frac{z_N\left(\left(\frac{1}{2}+\sigma\right) - j\omega,\rho(t)\right)}{\left(\frac{1}{2}+\sigma\right) - j\omega} \right] = \phi(\frac{1}{2} + j \left(\omega + j \sigma\right))
\end{eqnarray}
which will be used to investigate the Riemann Hypothesis and is stated in the following lemma:

\begin{lemma}
	\label{lm3}
 If the completed Riemann zeta function $\phi(\frac{1}{2} + j \left(\omega + j \sigma\right)) = 0$, then almost surely $\sigma = 0$.
\end{lemma}

\begin{proof}
We will use the notation 
\begin{eqnarray}
	\nonumber
	z(-\sigma+j\omega,t) &=& \lim_{N \rightarrow \infty} \lim_{t \rightarrow \infty} z_N(-\sigma+j\omega,t) = \lim_{N \rightarrow \infty} \lim_{t \rightarrow \infty} z_N\left(\left(\frac{1}{2}-\sigma\right) + j\omega,\rho(t)\right) \\ \label{lm3:eq:zlimits}
	z(\sigma-j\omega,t) &=& \lim_{N \rightarrow \infty} \lim_{t \rightarrow \infty} z_N(\sigma-j\omega,t) = \lim_{N \rightarrow \infty} \lim_{t \rightarrow \infty} z_N\left(\left(\frac{1}{2}+\sigma\right) - j\omega, \rho(t) \right).
\end{eqnarray}

Then, setting $\phi\left(\frac{1}{2} + j \left(\omega + j \sigma\right)\right) = 0$ in equation~\ref{sec3:eq:dynamicalsystem}, we obtain from equation~\ref{sec4:eq:phiz} the following condition:
\begin{equation}
	\lim_{N\rightarrow\infty} \lim_{t \rightarrow \infty}	\frac{\left|z_N(-\sigma+ j\omega,t)\right|}{\left|z_N(\sigma-j\omega,t)\right|} = \frac{\left|\left(\frac{1}{2}-\sigma\right) + j\omega\right|}{\left|\left(\frac{1}{2}+\sigma\right) - j\omega\right|}
	\label{lm3:eq:derv2}
\end{equation}
where $\left|.\right|$ denotes the magnitude of the complex variable. The RHS of~\ref{lm3:eq:derv2} is bounded since $-\frac{1}{2} < \sigma < \frac{1}{2}$. Therefore, the asymptotic behavior of the LHS in equation~\ref{lm3:eq:derv2} should also be bounded for all non-trivial zeros.

Using the definition of $z_N(-\sigma+j\omega,t)$ and $z_N(\sigma-j\omega,t)$ according to equations~\ref{sec3:eq:alphafn} and~\ref{sec3:eq:betafn} as
\begin{eqnarray}
	z_N(-\sigma+j\omega,t) &=& \mathcal{E}\left[\left| x_1 - x_2 \right|;(0,0,e^{-t}) \right] \\ \nonumber
	&+& \sum_{n=1}^N \mathcal{E}\left[\left| x_1 - x_2 \right|;(0,\alpha_n(t),e^{-t}) + (0,-\alpha_n(t),e^{-t}) + (\alpha_n(t),0,e^{-t}) + (-\alpha_n(t),0,e^{-t}) \right] \\ \nonumber
	z_N(\sigma-j\omega,t) &=& \mathcal{E}\left[\left| x_1 - x_2 \right|;(0,0,e^{-t}) \right] \\ \nonumber
&+& \sum_{n=1}^N \mathcal{E}\left[\left| x_1 - x_2 \right|;(0,\beta_n(t),e^{-t}) + (0,-\beta_n(t),e^{-t}) + (\beta_n(t),0,e^{-t}) + (-\beta_n(t),0,e^{-t}) \right]	
\end{eqnarray}
where $\alpha_n(t)$ and $\beta_n(t)$ are defined in equations~\ref{sec3:eq:alphalogcenter} and ~\ref{sec3:eq:betalogcenter}.

Substituting $\Delta x = x_1 - x_2$ leads to 
\begin{eqnarray}
\label{lm3:eq:derv3}
	z_N(-\sigma+j\omega,t) &=& \frac{e^{t/2}}{\sqrt{\pi}} \int_{-\infty}^{\infty}|\Delta x| \exp\left[-\frac{1}{4}e^{t}\left(\Delta x \right)^2\right] d\Delta x \\
	&+&\frac{2e^{t/2}}{\sqrt{\pi}} \sum_{n=1}^N \left( \int_{-\infty}^{\infty} |\Delta x| \exp\left[-\frac{1}{4}e^{t}\left(\Delta x + \alpha_n(t)\right)^2\right] d\Delta x \right)\\
	&+& \frac{2e^{t/2}}{\sqrt{\pi}} \sum_{n=1}^N \left( \int_{-\infty}^{\infty} |\Delta x| \exp\left[-\frac{1}{4}e^{t}\left(\Delta x - \alpha_n(t)\right)^2\right] d\Delta x \right)
\end{eqnarray}
Note that $\alpha_n(t) = \mathfrak{Re}(\alpha_n(t)) + j\mathfrak{Im}(\alpha_n(t))$ are complex functions with respect to time where using equation~\ref{sec4:eq:alphalog}
\begin{eqnarray}
	\label{lm3:eq:alphaasym}
	\mathfrak{Re}(\alpha_n(t)) &=&  n \sqrt{2\pi}  \exp\left(-\frac{1}{2} \left[1+\gamma_{\alpha}\right]t\right) \cos\left(\frac{1}{2}\omega_{\alpha}t \right) \\ \nonumber 
	\mathfrak{Im}(\alpha_n(t)) &=&  n \sqrt{2\pi}  \exp\left(-\frac{1}{2} \left[1+\gamma_{\alpha}\right]t\right) \sin\left(\frac{1}{2}\omega_{\alpha}t \right). 
\end{eqnarray}
We will consider the individual terms in the equation~\ref{lm3:eq:derv3}
\begin{equation}
	\label{lm3:eq:derv4}
	I_n(t) = e^{t/2}  \int_{-\infty}^{\infty} |\Delta x| \exp\left[-\frac{1}{4}e^{t}\left(\Delta x \pm \alpha_n(t)\right)^2\right] d\Delta x 
\end{equation}
which can be simplified as
\begin{eqnarray}
	\label{lm3:eq:derv5}
	I_n(t) &=& \exp\left[-\frac{1}{4}e^{t}\left(\alpha_n^2(t) - \mathfrak{Re}(\alpha_n(t))^2 \right)\right] \\ \nonumber
	&& e^{t/2}  \int_{-\infty}^{\infty} |\Delta x| \exp\left[\mp\frac{j}{2}e^{t}\Delta x \mathfrak{Im}(\alpha_n(t)) \right]\exp\left[-\frac{1}{4}e^{t}\left(\Delta x \pm \mathfrak{Re}(\alpha_n(t))\right)^2\right] d\Delta x. 
\end{eqnarray}

Asymptotically the term 
\begin{equation}
	\lim_{t \rightarrow \infty} \exp\left[-\frac{1}{4}e^{t}\left(\alpha_n^2(t) - \mathfrak{Re}(\alpha_n(t))^2 \right)\right] = 1
\end{equation}
where as the integral in equation~\ref{lm3:eq:derv5} represents the expectation of terms
\begin{equation}
	 |\Delta x| \exp\left[\mp\frac{j}{2}e^{t}\Delta x \mathfrak{Im}(\alpha_n(t)) \right] 
	\label{lm3:eq:asymp1}
\end{equation}
with respect to a Gaussian distribution
\begin{equation}
	p(\Delta x) = \frac{1}{2\sigma}\exp\left(-\frac{\left(\Delta x - \mu\right)^2}{4\sigma^2}\right)
	\label{lm3:eq:envp}
\end{equation}
with a mean $\mu = \mathfrak{Re}(\alpha_n(t))$ and a variance $\sigma^2 = \sqrt{2}e^{-t}$. Note that as $t \rightarrow \infty$, $\sigma^2 \rightarrow 0$, implying that the value of $\Delta x$ becomes concentrated around $\mathfrak{Re}(\alpha_n(t))$. Thus, for $\epsilon >0$, 
\begin{equation}
	\mathbb{P}\left(\left|\Delta x - \mathfrak{Re}(\alpha_n(t))\right| < \epsilon\right) \ge 1 - 2\exp\left[-\frac{\epsilon^2}{2} e^{t}\right]. 
\end{equation}
or 
\begin{equation}
	\lim_{t \rightarrow \infty} \mathbb{P}\left(\left|\Delta x - \mathfrak{Re}(\alpha_n(t))\right| < \epsilon\right) = 1. 
\end{equation}
for arbitrarily small $\epsilon$. Because $\gamma_{\alpha} > 0$,  equations~\ref{lm3:eq:alphaasym} leads to
\begin{equation}
	\lim_{t \rightarrow \infty} I_n(t) =  \lim_{t \rightarrow \infty} \left|\mathfrak{Re}(\alpha_n(t))\right|\exp\left[\mp\frac{j}{2}e^{t}\mathfrak{Re}(\alpha_n(t)) \mathfrak{Im}(\alpha_n(t)) \right] = \lim_{t \rightarrow \infty}\left|\mathfrak{Re}(\alpha_n(t))\right|.
\end{equation}
which follows from $\lim_{t \rightarrow \infty} \exp\left[\mp\frac{j}{2}e^{t}\mathfrak{Re}(\alpha_n(t)) \mathfrak{Im}(\alpha_n(t)) \right] = 1$.
Thus, with probability one, the following asymptotic result holds
\begin{equation}
	\label{lm3:eq:limitalpha}
	\lim_{N \rightarrow \infty} \lim_{t \rightarrow \infty} \frac{\left|z_N\left(-\sigma+j\omega,t\right)\right|}{\left|z_N\left(\sigma-j\omega,t\right)\right|} =  \lim_{N \rightarrow \infty} \lim_{t \rightarrow \infty} \frac{\sum_{n=1}^N\left|\mathfrak{Re}(\alpha_n(t))\right|}{\sum_{n=1}^N\left|\mathfrak{Re}(\beta_n(t))\right|} = \lim_{t \rightarrow \infty} e^{2\left[\gamma_{\beta} - \gamma_{\alpha}\right]t} \frac{\left|\cos \frac{1}{2}\omega_{\alpha}t\right|}{\left|\cos\frac{1}{2}\omega_{\beta}t\right|}.
\end{equation}
which from equation~\ref{lm3:eq:derv2} leads to
\begin{equation}
	 \lim_{t \rightarrow \infty} e^{2\left[\gamma_{\beta} - \gamma_{\alpha}\right]t} \frac{\left|\cos \frac{1}{2}\omega_{\alpha}t\right|}{\left|\cos \frac{1}{2}\omega_{\beta}t\right|} = \frac{\left(\frac{1}{2}-\sigma\right)^2 + \omega^2}{\left(\frac{1}{2}+\sigma\right)^2 + \omega^2}.
	\label{lm3:eq:asympRHpre}
\end{equation}
The LHS of equation~\ref{lm3:eq:asympRHpre} is bounded if and only if $\gamma_{\beta} = \gamma_{\alpha}$ and $\omega_{\alpha} = \omega_{\beta}$. From equations~\ref{sec4:eq:alphaparam} and~\ref{sec4:eq:betaparam} the condition is satisfied only if  $\sigma = 0$. {\bf Hence, R.H. is almost surely true}

\end{proof}

\section{Extension to Dirichlet L-functions}
\label{section5}
We now show that the dynamical systems framework in section~\ref{section4} can also be extended to study the non-trivial zeros of the Dirichlet L-functions. The Dirichlet L-function~\citep{dirichlet}  is defined as an analytic continuation of the Dirichlet series $L:\mathbb{C}\times\mathbb{C} \rightarrow \mathbb{C}$
\begin{equation}
	L(s,\chi) = \sum_{n=1}^{\infty} \frac{\chi(n)}{n^s}; \quad \mathfrak{Re}(s) > 1
	\label{sec5:eq:defn}
\end{equation}
where $\chi: \mathbb{Z} \rightarrow \mathbb{C}$ represents a Dirichlet character modulo $P$ with the following properties
\begin{enumerate}
	\item $\chi(m \times n) = \chi(m)\chi(n)$ for every $m,n \in \mathbb{Z}$.
	\item $\chi(n+P) = \chi(n)$ for all $n$.
	\item $\chi(n) = 0$ if gcd$(n,q) > 1$.
\end{enumerate}
Note that for a trivial primitive Dirichlet character $\chi(n)=1$ modulo $P$, the Dirichlet L-function reduces to the Riemann zeta-function.
Even though the series $L(s,\chi)$ is defined over the domain $\mathfrak{Re}(s) > 1$ it can be extended over the entire s-domain provided $\chi$ is primitive modulo $P$~\citep{dirichlet}. Like in sections~\ref{section3} and~\ref{section4}, generating the Dirichlet L-function will exploit the analytic continuation procedure, but we will restrict the domain only to the critical-strip. We will use this procedure to investigate the Generalized Riemann hypothesis (G.R.H.)~\citep{dirichlet,titchmarsh} which states that {\it all non-trivial zeros of $L(s,\chi)$ lie on the axis $\mathfrak{Re}(s)=\frac{1}{2}$}. However, we will consider two cases $L(s,\chi)$ of depending on whether $\chi$ is even or odd.

\subsection{Dirichlet L-function with even parity $\chi$} 
For an even parity $\chi$ function satisfying
\begin{equation}
	\chi(-n) = \chi(n)
	\label{sec5eq:evenparity}
\end{equation}
a completed Dirichlet L-function $\phi_{\chi}(s)$ can be defined like its Riemann zeta function analogue in an integral form as
\begin{equation}
	\phi_{\chi}^e(s):=\pi^{-\frac{s}{2}}\Gamma\left(\frac{s}{2}\right)L(s,\chi) = \int_{0}^{\infty} y^{\frac{s}{2}-1} \Psi_{\chi}(y) dy 
	\label{sec5:eq:evenintegral}
\end{equation}
where $\Psi_{\chi}$ is given by
\begin{equation}
	\Psi_{\chi}(y) =\sum_{n=1}^{\infty} \chi(n) \exp\left[-\pi n^2 y\right].
	\label{sec5:eq:evenPsi}
\end{equation}
The analytic continuation of equation~\ref{sec5:eq:evenintegral} follows a similar approach as the analytic continuation of Riemman zeta function where the completed Dirichlet L-function can be written as
\begin{equation}
	\phi_{\chi}^e(s) = \int_{\frac{1}{P}}^{\infty} y^{\frac{s}{2}-1} \Psi_{\chi}(y) dy + \frac{P^{1-s}}{G\left(\overline{\chi},P\right)} \int_{\frac{1}{P}}^{\infty} y^{\frac{1-s}{2}-1} \Psi_{\overline{\chi}}(y) dy
	\label{sec5:eq:evenanal}
\end{equation}
where $G\left(\chi,P\right)$ denotes a Gauss sum
\begin{equation}
	G\left(\chi,P\right) = \sum_{n = 1}^P \chi(n) \e^{-2 \pi j \frac{n}{P}}
	\label{sec5:eq:gauss}
\end{equation}
and $\overline{\chi}$ denotes the complex conjugate or inverse of the Dirichlet character. The analytic continuation in~\ref{sec5:eq:evenanal} is now valid over the entire s-domain. 

For our analysis, we will restrict the domain to $\mathfrak{Re}(s) \in (0,1)$ and will use the following relation transformation
\begin{equation}
	\int_{0}^{\frac{1}{P}} y^{\frac{s}{2}-1} \Psi_{\chi}(y) dy = P^{-\frac{s}{2}}	\int_{0}^{1} y^{\frac{s}{2}-1} \Psi_{\chi}\left(\frac{y}{P}\right) dy.
	\label{sec5:eq:evenstep2}
\end{equation}
and then working backwards, like in section~\ref{section3}, we could construct symmetric generator functions $z(s)$ and $z(1-s)$ according to 
\begin{eqnarray}
	\label{sec5:eq:evenalpha}
	z(s) &=& \frac{4s}{\sqrt{\pi}}\int_{0}^{1} y^{\frac{s}{2}-1} \Psi_{\chi}\left(\frac{y}{P}\right) dy \\  \label{sec5:eq:evenbeta}
	z(1-s) &=& \frac{4(1-s)}{\sqrt{\pi}}\int_{0}^{1} y^{\frac{1-s}{2}-1} \Psi_{\overline{\chi}}\left(\frac{y}{P}\right) dy	
\end{eqnarray}
which after using equation~\ref{sec5:eq:evenstep2} will lead to
\begin{eqnarray}
	\phi_{\chi}^e(s) &+& \frac{\sqrt{\pi}}{4} \left[P^{-\frac{s}{2}} \frac{z(s)}{s} + \frac{P^{-\frac{1-s}{2}}}{G\left(\overline{\chi},P\right)}\frac{z(1-s)}{1-s} \right] \\
	&=& \int_{0}^{\infty} y^{\frac{s}{2} - 1} \Psi_{\chi}(y) dy + \frac{P^{-\frac{1-s}{2}}}{G\left(\overline{\chi},P\right)}\int_{0}^{\infty} y^{\frac{1-s}{2} - 1} \Psi_{\overline{\chi}}(y) dy \\ \label{sec5:eq:analytic1}
	&=& \phi_{\chi}^e(s) + \frac{P^{1-s}}{G\left(\overline{\chi},P\right)} \phi_{\overline{\chi}}(1-s) \\
	&=& 2\phi_{\chi}^e(s)
\end{eqnarray}
Thus,
\begin{equation}
	\phi_{\chi}^e(s) = \frac{\sqrt{\pi}}{4} \left[P^{-\frac{s}{2}} \frac{z(s)}{s} + \frac{P^{-\frac{1-s}{2}}}{G\left(\overline{\chi},P\right)}\frac{z(1-s)}{1-s} \right]; \quad \mathfrak{Re}(s) \in (0,1)
	\label{sec5:eq:analyticfinal}
\end{equation}

Like the section~\ref{section3}, the symmetric functions $z_N(s,\rho)$ and $z_N(1-s,\rho)$ can be generated using the functions
 \begin{eqnarray}
	\label{sec5:eq:evenalphafn}
&&	z_N(s,\rho) = \sum_{n=1}^{N} \chi(n) \left(\mathcal{E}\left[\left|x_1 - x_2\right|;(\alpha_n,0,\rho)+(-\alpha_n,0,\rho)+(0,\alpha_n,\rho)+(0,-\alpha_n,\rho)\right] \right) \\ \nonumber
&&	z_N(1-s,\rho) = \sum_{n=1}^{N}\overline{\chi}(n) \left(\mathcal{E}\left[\left| \tilde{x}_1 - \tilde{x}_2 \right|;(\beta_n,0,\rho)+ (-\beta_n,0,\rho)+(0,\beta_n,\rho)+ (0,-\beta_n,\rho)\right] \right). \\ \nonumber
\end{eqnarray}
The parameters $\{\alpha_n\}$ and $\{\beta_n\}$ are modified to account for the conductor $P$ in the Dirichlet character as $\alpha_n \rightarrow \alpha_n/\sqrt{P}$ and $\beta_n \rightarrow \beta_n/\sqrt{P}$. 

Now we can state the following Lemma related to G.R.H.
\begin{lemma}
	\label{lm4}
	For even parity Dirichlet character $\chi$, all the non-trivial solutions satisfying $L(\frac{1}{2} + j \left(\omega + j \sigma\right),\chi) = 0$ also satisfies $\sigma = 0$ almost surely.
\end{lemma}

\begin{proof}
	The proof the Lemma follows similar definitions and procedure as in section~\ref{section3} and Lemma~\ref{lm3}. Setting the RHS of equation~\ref{sec5:eq:analyticfinal} to zero leads to	
	\begin{equation}
		\frac{\left| z\left(\left(\frac{1}{2}-\sigma\right) + j\omega\right)\right|}{\left|z\left(\left(\frac{1}{2}+\sigma\right) - j\omega \right)\right|} = \frac{\left|P^{\frac{1}{2}}\right|}{\left|G\left(\overline{\chi},P\right)\right|} \frac{\left|\left(\frac{1}{2}-\sigma\right) + j\omega\right|}{\left|\left(\frac{1}{2}+\sigma\right) - j\omega\right|}
		\label{lm4:eq:derv1}
	\end{equation}
	Since $\chi$ is primitive modulo $P$, the magnitude of the Gauss sum $\left|G\left(\overline{\chi},P\right)\right| = \sqrt{P}$. Hence, equation~\ref{lm4:eq:derv1} can be simplified to
	\begin{equation}
		\frac{\left| z\left(\left(\frac{1}{2}-\sigma\right) + j\omega\right)\right|}{\left|z\left(\left(\frac{1}{2}+\sigma\right) - j\omega \right)\right|} = \frac{\left|\left(\frac{1}{2}-\sigma\right) + j\omega\right|}{\left|\left(\frac{1}{2}+\sigma\right) - j\omega\right|}
		\label{lm4:eq:derv2}
	\end{equation}
	which is identical to the equation~\ref{lm3:eq:derv2} in Lemma~\ref{lm3}. Thus,following a similar asymptotic analysis as in Lemma~\ref{lm3} leads to
	\begin{equation}
	 \lim_{N\rightarrow \infty}\left(\frac{\left|\sum_{n=1}^N n \chi(n)\right|}{\left|\sum_{n=1}^N n \overline{\chi}(n)\right|}\right) \lim_{t \rightarrow \infty} e^{2\left[\gamma_{\beta} - \gamma_{\alpha}\right]t} \frac{\left|\cos \frac{1}{2}\omega_{\alpha}t\right|}{\left|\cos \frac{1}{2}\omega_{\beta}t\right|} = \frac{\left(\frac{1}{2}-\sigma\right)^2 + \omega^2}{\left(\frac{1}{2}+\sigma\right)^2 + \omega^2}.
		\label{lm4:eq:asymp3}
	\end{equation}
 Since, both $\chi$ and $\overline{\chi}$ are periodic with respect to $P$ and
	$\left|\chi\right| = \left|\overline{\chi}\right|$, the term
	\begin{equation}
		\lim_{N\rightarrow \infty}\left(\frac{\left|\sum_{n=1}^N n \chi(n)\right|}{\left|\sum_{n=1}^N n \chi^{-1}(n)\right|}\right) = A
	\end{equation}
	is bounded from above. Thus, like Lemma~\ref{lm3} equation~\ref{lm4:eq:asymp3} holds almost surely if $\sigma = 0$ and {\em hence G.R.H. is almost surely true}. 
\end{proof}

\subsection{Dirichlet L-function with odd parity $\chi$ } 
For an odd parity the function $\chi$ satisfy
\begin{equation}
	\chi(-n) = -\chi(n)
	\label{sec5:eq:oddparity}
\end{equation}
and a completed Dirichlet L-function can be defined in an integral form as
\begin{equation}
	\phi_{\chi}^o(s) := \pi^{-\frac{s+1}{2}}\Gamma\left(\frac{s+1}{2}\right)L(s,\chi) = \int_{0}^{\infty} y^{\frac{s}{2}-1} \tilde{\Psi}_{\chi}(y) dy 
	\label{sec5:eq:oddintegral}
\end{equation}
where $\tilde{\Psi}_{\chi}$ is defined to account for the odd parity as
\begin{equation}
	\tilde{\Psi}_{\chi}(y) =\sum_{n=1}^{\infty}\chi(n) n \sqrt{y} \exp\left[-\pi n^2 y\right].
	\label{sec5:eq:oddPsi}
\end{equation}
The analytic continuation of equation~\ref{sec5:eq:oddintegral} follows a similar approach as the analytic continuation of the even parity case, where the completed Dirichlet L-function can be written as
\begin{equation}
	\phi_{\chi}^o(s) = \int_{\frac{1}{P}}^{\infty} y^{\frac{s}{2}-1} \tilde{\Psi}_{\chi}(y) dy + \frac{-jP^{1-s}}{G\left(\overline{\chi},P\right)} \int_{\frac{1}{P}}^{\infty} y^{\frac{1-s}{2}-1} \tilde{\Psi}_{\overline{\chi}}(y) dy
	\label{sec5:eq:oddanalytic}
\end{equation}
The analytic continuation in~\ref{sec5:eq:oddanalytic} is now valid over the entire s-domain. 

Again, for our analysis, we will restrict the domain to $\mathfrak{Re}(s) \in (0,1)$ and will use the following relation 
\begin{equation}
	\int_{0}^{\frac{1}{P}} y^{\frac{s}{2}-1} \tilde{\Psi}_{\chi}(y) dy = P^{-\frac{s}{2}}	\int_{0}^{1} y^{\frac{s}{2}-1} \tilde{\Psi}_{\chi}\left(\frac{y}{P}\right) dy.
	\label{sec5:eq:oddstep2}
\end{equation}

Working backwards, like in section~\ref{section3}, we could construct symmetric generator functions $z(s,\rho)$ and $z(1-s,\rho)$ according to 
\begin{eqnarray}
	\label{sec5:eq:oddalpha}
	z(s,\rho) &=& \frac{4(s+1)}{\sqrt{\pi}}\int_{0}^{1} y^{\frac{s}{2}-1} \tilde{\Psi}_{\chi}\left(\frac{y}{P}\right) dy \\  \label{sec5:eq:oddbeta}
	z(1-s,\rho) &=& \frac{4(2-s)}{\sqrt{\pi}}\int_{0}^{1} y^{\frac{1-s}{2}-1} \tilde{\Psi}_{\overline{\chi}}\left(\frac{y}{P}\right) dy	
\end{eqnarray}
which after using equation~\ref{sec5:eq:oddstep2} will lead to
\begin{eqnarray}
	\phi_{\chi}^o(s) &+& \frac{\sqrt{\pi}}{4} \left[P^{-\frac{s}{2}} \frac{z(s)}{s+1} -j \frac{P^{-\frac{1-s}{2}}}{G\left(\overline{\chi},P\right)}\frac{z(1-s)}{2-s} \right]  \\
	&=& \int_{0}^{\infty} y^{\frac{s}{2} - 1} \tilde{\Psi}_{\chi}(y) dy + \frac{P^{-\frac{1-s}{2}}}{G\left(\overline{\chi},P\right)}\int_{0}^{\infty} y^{\frac{1-s}{2} - 1} \tilde{\Psi}_{\overline{\chi}}(y) dy \\ \label{sec5:eq:oddanalytic1}
	&=& \phi_{\chi}^o(s) -j \frac{P^{1-s}}{G\left(\overline{\chi},P\right)} \phi_{\overline{\chi}}(1-s) \\
	&=& 2\phi_{\chi}^o(s)
\end{eqnarray}
Thus,
\begin{equation}
	\phi_{\chi}^o(s) = \frac{\sqrt{\pi}}{4} \left[P^{-\frac{s}{2}} \frac{z(s)}{1+s} -j \frac{P^{-\frac{1-s}{2}}}{G\left(\overline{\chi},P\right)}\frac{z(1-s)}{2-s} \right]; \quad \mathfrak{Re}(s) \in (0,1)
	\label{sec5:eq:oddanalyticfinal}
\end{equation}
Note that the factors $s+1$ and $2-s$ in equations~\ref{sec5:eq:oddanalytic1}-~\ref{sec5:eq:oddanalyticfinal} is different from the even parity case and accounts for the modified $\tilde{\Psi}_{\chi}(.)$ in equation~\ref{sec5:eq:oddPsi}. 
Like the section~\ref{section3}, the symmetric functions $z_N(s,\rho)$ and $z_N(1-s,\rho)$ can be generated according to
\begin{eqnarray}
	\label{sec5:eq:oddalphafn}
	&&	z_N(s,\rho) = \sum_{n=1}^{N} \chi(n) n \left(\mathcal{E}\left[\left|x_1 - x_2\right|;(\alpha_n,0,\rho)+(-\alpha_n,0,\rho)+(0,\alpha_n,\rho)+(0,-\alpha_n,\rho)\right] \right) \\ \nonumber
	&&	z_N(1-s,\rho) = \sum_{n=1}^{N}\overline{\chi}(n) n \left(\mathcal{E}\left[\left| \tilde{x}_1 - \tilde{x}_2 \right|;(\beta_n,0,\rho)+ (-\beta_n,0,\rho)+(0,\beta_n,\rho)+ (0,-\beta_n,\rho)\right] \right) \\ \nonumber
\end{eqnarray}
where the parameters $\{\alpha_n\}$ and $\{\beta_n\}$ are now different compared to the even parity case with respect to the parameters $\gamma_{\alpha},\omega_{\alpha},\gamma_{\beta},\omega_{\beta}$. These parameters for the odd-parity Dirichlet L-function are defined as 
\begin{eqnarray}
	\label{sec5:eq:oddalphagamma}
	\gamma_{\alpha} &=& \frac{\left(\frac{3}{2} - \sigma\right) }{\left(\frac{3}{2} - \sigma\right)^2 + \omega^2} \\ \label{sec5:eq:oddalphaomega}	
	\omega_{\alpha} &=&  \frac{\omega}{\left(\frac{3}{2} - \sigma\right)^2 + \omega^2} . 
\end{eqnarray}
and
\begin{eqnarray}
	\gamma_{\beta} &=& \frac{\left(\frac{3}{2} + \sigma\right) }{\left(\frac{3}{2} + \sigma\right)^2 + \omega^2} \\ \label{sec5:eq:oddalphaomegad}	
	\omega_{\beta} &=&  \frac{\omega}{\left(\frac{3}{2} + \sigma\right)^2 + \omega^2} 	
\end{eqnarray}

We can state the following Lemma.
\begin{lemma}
	\label{lm5}
	For an odd parity Dirichlet character $\chi$, all the non-trivial solutions satisfying $L(\frac{1}{2} + j \left(\omega + j \sigma\right),\chi) = 0$ also satisfies $\sigma = 0$ almost surely.
\end{lemma}

\begin{proof}
	The proof the Lemma follows similar definitions and procedure as in section~\ref{section3} and for the even parity case. Setting the LHS of equation~\ref{sec5:eq:oddanalyticfinal} to zero leads to	
	\begin{equation}
		\frac{\left| z\left(\left(\frac{1}{2}-\sigma\right) + j\omega\right)\right|}{\left|z\left(\left(\frac{1}{2}+\sigma\right) - j\omega \right)\right|} = \frac{\left|j P^{\frac{1}{2}}\right|}{\left|G\left(\overline{\chi},P\right)\right|} \frac{\left|\left(\frac{3}{2}-\sigma\right) + j\omega\right|}{\left|\left(\frac{3}{2}+\sigma\right) - j\omega\right|}
		\label{lm5:eq:derv1}
	\end{equation}
	Since $\chi$ is a primitive modulo $P$, the magnitude of the Gauss sum $\left|G\left(\overline{\chi},P\right)\right| = \sqrt{P}$. Hence, equation~\ref{lm5:eq:derv1} can be simplified to
	\begin{equation}
		\frac{\left| z\left(\left(\frac{1}{2}-\sigma\right) + j\omega\right)\right|}{\left|z\left(\left(\frac{1}{2}+\sigma\right) - j\omega \right)\right|} = \frac{\left|\left(\frac{3}{2}-\sigma\right) + j\omega\right|}{\left|\left(\frac{3}{2}+\sigma\right) - j\omega\right|}
		\label{lm5:eq:derv2}
	\end{equation}
	which is similar to the Lemma~\ref{lm4} and for the even parity case except for the symmetry about $\frac{3}{2}$ instead of $\frac{1}{2}$ on the RHS. However, the RHS is still bounded and for $\sigma = 0$ the RHS of~\ref{lm5:eq:derv2} still equals unity. 
	
	Following similar arguments as equations~\ref{lm3:eq:asymp1}-~\ref{lm3:eq:asympRHpre} and using concentration bounds leads to the statement that with a probability greater than $1 - \epsilon$, there exists an $n_0 > 0$ such that
	\begin{equation}
	 \lim_{N\rightarrow \infty}\left(\frac{\left|\sum_{n=1}^N n^2 \chi(n)\right|}{\left|\sum_{n=1}^N n^2 \overline{\chi}(n)\right|}\right) \lim_{t \rightarrow \infty} e^{2\left[\gamma_{\beta} - \gamma_{\alpha}\right]t} \frac{\left|\cos \frac{1}{2}\omega_{\alpha}t\right|}{\left|\cos \frac{1}{2}\omega_{\beta}t\right|} = \frac{\left(\frac{3}{2}-\sigma\right)^2 + \omega^2}{\left(\frac{3}{2}+\sigma\right)^2 + \omega^2}
		\label{sec5:eq:oddasymp1}
	\end{equation}
	holds. Since, both $\chi$ and $\overline{\chi}$ are periodic with respect to $P$ and
	$\left|\chi\right|$ and $\left|\overline{\chi}\right|$ the term
	\begin{equation}
		\lim_{N\rightarrow \infty}\left(\frac{\left|\sum_{n=1}^N n^2 \chi(n)\right|}{\left|\sum_{n=1}^N n^2 \overline{\chi}(n)\right|}\right) = A 
	\end{equation}
	is bounded from above implying that the Lemma holds almost surely, which is the statement for the G.R.H. for the Dirichlet L-function with odd parity.
\end{proof}

\section{Discussions and Conclusions}
\label{discussions}

In this paper we used an extension of the Price's theorem to construct a novel dynamical systems model whose steady-state evolves towards the completed Riemann zeta function and in a more general case towards a Dirichlet L-function. 
The model describes a dynamical system that combines both stochastic and deterministic processes where the two random sources become more correlated and asymptotically cancel each other out resulting in an evolution towards a deterministic steady-state function. In this paper we chose a specific form of the dynamical system that resulted in the completed Riemann zeta function and the Dirichlet L-function as the steady-state. However, choosing different forms of non-linear function $z_N$ and the variables $\alpha_n, \beta_n, n = 1,2,..$, should result in other dynamical systems formulation that evolve towards other functions. We believe that any Dirichlet series function that admits analytic continuation similar to the steps in~\ref{sec5:eq:analytic1} and~\ref{sec5:eq:analyticfinal}, can be analyzed using this model.

\section{Appendix I: Proof of Lemma 1}
\label{appendix1}

	The proof of Lemma~\ref{lm1} follows a similar procedure as the Price's theorem and its extensions~\citep{price,mcmahon,papoulis}, both of which use the Fourier transform $M(\omega_1,\omega_2)$ defined in equation~\ref{sec2:eq:ampfourier}. Following an interchange of variables leads to 
	\begin{eqnarray}
		& &\mathcal{E}\left[f;(m_1,m_2,\rho)\right] = \int_{-\infty}^{\infty} \int_{-\infty}^{\infty} f(x_1,x_2)p\left[x_1,x_2;(m_1,m_2,\rho)\right]dx_1 dx_2. \\ \label{lm1:eq:derv1}
		&=& \int_{-\infty}^{\infty} \int_{-\infty}^{\infty} f(x_1,x_2) \left[\frac{1}{\left(2\pi j\right)^2}
		\int_{-\infty}^{\infty} \int_{-\infty}^{\infty} M(\omega_1,\omega_2) \e^{- j \omega_1 x_1 - j \omega_2 x_2 } d\omega_1 d\omega_2 \right]dx_1 dx_2.
	\end{eqnarray}
	Using equations~\ref{lm1:eq:derv1} and~\ref{sec2:eq:ampfourier}, the following steps
	\begin{eqnarray}
		\nonumber
		& &\frac{\partial \mathcal{E}\left[f;(m_1,m_2,\rho)\right]}{\partial \rho} \\ \nonumber
		&=& \int_{-\infty}^{\infty} \int_{-\infty}^{\infty} f(x_1,x_2) \left[\frac{1}{\left(2\pi j\right)^2}
		\int_{-\infty}^{\infty} \int_{-\infty}^{\infty} \frac{\partial M(\omega_1,\omega_2)}{\partial \rho} \e^{- j \omega_1 x_1 - j \omega_2 x_2} d\omega_1 d\omega_2 \right]dx_1 dx_2 \\ \nonumber 
		&=& \int_{-\infty}^{\infty}\int_{-\infty}^{\infty} \frac{1}{\left(2\pi j\right)^2} M(\omega_1,\omega_2) \left(\omega_1\omega_2 + j\frac{d m_1}{d \rho} \omega_1+ j\frac{d m_2}{d \rho}\omega_2 \right) \\ \nonumber
		& & \quad \quad \quad \quad \left[\int_{-\infty}^{\infty}\int_{-\infty}^{\infty} f(x_1,x_2) 
		\e^{- j \omega_1 x_1 - j \omega_2 x_2} dx_1 dx_2 \right] d\omega_1 d\omega_2 \\ \nonumber
		&=& \int_{-\infty}^{\infty} \int_{-\infty}^{\infty} \left[\frac{\partial^2 f}{\partial x_1 \partial x_2} + \frac{d m_1}{d \rho}\frac{\partial f}{\partial x_1} + \frac{d m_2}{d \rho}\frac{\partial f}{\partial x_2} \right] \\ \nonumber
		& & \quad \quad \quad \quad \left[\frac{1}{\left(2\pi j\right)^2}
		\int_{-\infty}^{\infty} \int_{-\infty}^{\infty} M(\omega_1,\omega_2) \e^{- j \omega_1 x_1 - j \omega_2 x_2} d\omega_1 d\omega_2 \right] dx_1 dx_2 \\ \label{lm1:eq:derv2}
		&=& \int_{-\infty}^{\infty} \int_{-\infty}^{\infty} \left[\frac{\partial^2 f}{\partial x_1 \partial x_2} + \frac{d m_1}{d \rho}\frac{\partial f}{\partial x_1} + \frac{d m_2}{d \rho}\frac{\partial f}{\partial x_2} \right] p\left[x_1,x_2;(m_1,m_2,\rho)\right] dx_1 dx_2 		
	\end{eqnarray}
	proves the lemma~\ref{lm1}.

\end{document}